\documentclass[11pt]{amsart}
\usepackage{amscd}
\usepackage{amsfonts}
\usepackage{color}
\usepackage{amsmath,amssymb,amsthm,latexsym}
\usepackage{amscd}

\usepackage{multicol}
\usepackage{enumerate}

\newtheorem{theorem}{Theorem}[section]
\newtheorem{proposition}[theorem]{Proposition}
\newtheorem{definition}[theorem]{Definition}
\newtheorem{corollary}[theorem]{Corollary}
\newtheorem{lemma}[theorem]{Lemma}
\newtheorem{problem}[theorem]{Problem}
\usepackage{amsmath}
\usepackage{amsfonts}
\usepackage{amsmath,amsthm}
\usepackage{amscd}
\usepackage[all]{xy}
\numberwithin{equation}{section}
\theoremstyle{remark}

\newtheorem{example}[theorem]{\bf Example}
\newcommand{\R}{\mathbb{R}}

\newcommand{\dd}{\mathrm{d}}

\newcommand{\Lf}{\mathcal{L}^f}
\newcommand{\Lfa}{\mathcal{L}^{f_1}}
\newcommand{\Lfb}{\mathcal{L}^{\hat{f}_1}}

\begin{document}

\title[Morse index of minimal products]{\bf{ Morse index of minimal products of minimal submanifolds in spheres}}
\author{Changping Wang, Peng Wang}

\address{School of Mathematics and Statistics, FJKLMAA, Fujian Normal University, Fuzhou 350117, P. R. China}
\email{cpwang@fjnu.edu.cn}
\address{School of Mathematics and Statistics, FJKLMAA, Fujian Normal University, Fuzhou 350117, P. R. China}
\email{pengwang@fjnu.edu.cn}

\begin{abstract} Tang-Zhang \cite{TZ}, Choe-Hoppe \cite{CH}, showed independently that one can produce minimal submanifolds in spheres via Clifford type minimal product of minimal submanifolds. In this note, we show  that the  minimal product is immersed by its first eigenfunctions (of its Laplacian) if and only if the two beginning minimal submanifolds are immersed by their first eigenfunctions.  Moreover, we give estimates of Morse index and nullity of the minimal product. In particular, we show that the Clifford  minimal submanifold $\left(\sqrt{\frac{n_1}{n}}S^{n_1},\cdots,\sqrt{\frac{n_k}{n}}S^{n_k}\right)\subset S^{n+k-1}$ has index $(k-1)(n+k+1)$ and nullity $(k-1)\sum_{1\leq i<j\leq k}(n_i+1)(n_j+1)$ (with $n=\sum n_j$).
\end{abstract}

%\date{\today}
\maketitle

\vspace{0.5mm}{\bf \ \ ~~Keywords:}    Minimal product; Index; Nullity; Clifford minimal submanifold.
\vspace{0.5mm}

{\bf\ \  ~~ MSC(2020): \hspace{2mm} 53C40, 53A31, 53A30}

%\tableofcontents

\section{Introduction}

The study of minimal submanifolds in spheres plays an important role in the global differential geometry and geometric analysis and has a close relation with spectral geometry. It turns out to be a challenging problem for the study of  global  properties of minimal submanifolds in spheres,  see for instance \cite{Brendle,ElSoufi1,Gu,KW,Lawson,Li-y,Montiel,QT,Simons,TY2,TZ,Yau} and reference therein.

Recently Tang-Zhang \cite{TZ}, Choe-Hoppe \cite{CH} showed independently that one can produce minimal submanifolds in spheres via Clifford type product of minimal submanifolds. Moreover, Tang and Zhang used this construction for isoparametric hypersurfaces to discuss area-minimizing properties of minimal submanifolds \cite{TZ}.

It is natural to discuss further geometric and analytic properties of minimal products of minimal submanifolds from several directions. The first one concerns the spectrum properties of (the Laplacian of ) them in view of Yau's conjecture and a problem proposed recently by Tang and Yan \cite{TY2}. Recall that for minimal hypersurfaces, the well-known Yau's conjecture states that an embedded, closed minimal hypersurface in $S^{n+1}$ has $\lambda_1=n$ with $\lambda_1$ being the first eigenvalue of its Laplacian. We refer to \cite{Brendle}, \cite{CS} and \cite{TY2} for recent progress on this topic. On the other hand, an $n-$dimensional minimal submanifold is called immersed by its  first eigenfunctions if $\lambda_1=n$, which appears naturally in many important geometric problems,  for instance, Willmore conjecture \cite{Li-y,Montiel,ElSoufi1}  and $\lambda_1-$extremal metrics \cite{EGJ}. It is therefore natural to ask whether the minimal product of minimal submanifolds immersed by first eigenfunctions is also immersed by first eigenfunctions, which turns out to be true (see Theorem \ref{thm-first}). In this way, one can construct many new examples of minimal submanifolds immersed by first eigenfunctions of any dimension and any co-dimension greater than $2$, oriented or non-oriented. Moreover, it is not surprising that one can also construct examples of embedded minimal submanifolds {\em not }immersed by first eigenfunctions of higher co-dimension, which was first proposed as a problem in \cite[p. 526]{TY2}, concerning the influence of dimension and co-dimensions on embedded minimal submanifolds {\em immersed or not immersed}  by first eigenfunctions. Examples in terms of minimal products of some focal submanifolds in \cite[Proposition 1.1]{TY2} and totally geodesic spheres give embedded minimal submanifolds of dimension $\geq n+k+2$ and of co-dimension $\geq k+2$ for any $n\geq 1$ and $k\geq 3$. We refer to Section 3 for more details.

Then we focus on the estimate of Morse index and nullity of minimal submanifolds. It is well-known that the seminal work of J. Simons \cite{Simons} showed that all oriented closed minimal submanifolds in spheres are {\em not} stable. Then he gave a lower bound estimate of the Morse index and nullity of  minimal submanifolds in spheres, which turn out to be of great importance in the study of the geometric and analytical properties of minimal submanifolds. For example, the index characterization of Morse index of Clifford torus  \cite{Urbano} is one of the starting points in Marques and Neves' proof of the famous Willmore conjecture in $S^3$ \cite{Marques}. Recently, Kapouleas and Wiygul computed the Morse index and nullity of the Lawson minimal surface $\xi_{g,1}$ \cite{KW}.  We focus on estimates of the Morse index and nullity for minimal products in Section 4. In particular, we compute the Morse index and nullity of all Clifford minimal submanifolds in Proposition \ref{prop-CT} and \ref{prop-CT-N}.

Finally, we  consider the estimate of the length of the second fundamental form of minimal products, including the famous Simons' inequality and Chern's conjecture/problem \cite{Simons,CDK}. We refer to \cite{DX,Gu,TWY,TY3} for recent progress on this topic. For minimal products, we will derive a Simon's
type inequality and discuss briefly  Chern's problem for submanifolds.\vspace{2mm}

The paper is organized as follows: Section 2 contains the geometry of products of  submanifolds. In Section 3 we consider the first eigenvalue estimate of minimal product. Then in Section 4 the index and nullity estimates of minimal product are provided, as well as the index and nullity of all Clifford minimal submanifolds. In Section 5 we give some estimate of the square of the length of the second fundamental form of a minimal product.
We end this paper with an appendix for a proof of an index estimate of an $n-$dimensional minimal submanifold in $S^{n+p+1}$.

\section{The geometry of products of  submanifolds in spheres}
In this section we will collect some basic setup for products of  submanifolds in spheres.

\subsection{The basic setup}
First we will recall briefly the geometry of  products of  submanifolds in spheres, following basically the treatment in Section 4.1 of \cite{TZ} (here we use slightly different notations). Note that although they consider product of minimal submanifolds in \cite{TZ}, the computations work for  {products of arbitrary} submanifolds in spheres.

	Let $f_1:M_1\rightarrow S^{n_1+p_1}$  and  $\hat f_1:\hat M_1\rightarrow S^{\hat n_1+\hat p_1}$  be two  submanifoilds of dimension $n_1$ and $\hat{n}_1$ respectively. Set $n=n_1+\hat n_1$ and $p=p_1+\hat p_1+1$.
	Then $f=(c_1f_1,\hat{c}_1\hat{f}_{1}):M=M_1\times\hat  M_1\rightarrow S^{n+p}$, with $c_1^2+\hat{c}_1^2=1$ and $c_1\hat{c}_1\neq0$, is called  a {\em product } of $f_1$ and $\hat f_1$.

	 Let $\{\sigma_1,\cdots,\sigma_{p_1}\}$ and $\{\tau_1,\cdots,\tau_{\hat{p}_1}\}$ denote the  {local} orthonormal basis of the normal bundle of $f_1$ and $\hat {f}_1$  { around a point} respectively. Then an orthonormal basis of the normal bundle of $f$ is
\[(\sigma_1,0),\cdots,(\sigma_{p_1},0),(0,\tau_1),\cdots,(0,\tau_{\hat p_1}), \eta,\]
where
\[\eta:=(\hat c_1 f_1,-c_1 \hat f_1)\]
 {is a global unit normal vector field of $f$.}
Let $A^{f}$, $A^{f_1}$ and $A^{\hat f_1}$ denote the shape operators  of $f$, $f_1$ and $\hat{f}_{1}$ respectively.
Then  we have (Here we split the tangent bundle of $M$: $TM=TM_1\times T\hat{M}_1$.)
\begin{equation}\label{eq-A1}
A^f_{(\sigma_i,0)}=\left(
                       \begin{array}{cc}
                         \frac{1}{c_1} A^{f_1}_{\sigma_i}& 0 \\
                         0 & 0 \\
                       \end{array}
                     \right),
\end{equation}
and
\begin{equation}\label{eq-A2}A^f_{(0,\tau_j)}=\left(
                       \begin{array}{cc}
                         0 & 0 \\
                0&     \frac{1}{\hat c_1} A^{\hat f_1}_{\tau_j} \\
                       \end{array}
                     \right).
\end{equation}
For $\eta$, we have
\begin{equation}\label{eq-A-eta}
 A^f_{\eta}=\left(
                       \begin{array}{cc}
                         -\frac{\hat c_1}{c_1}I_{n_1} & 0 \\
                0&     \frac{ c_1}{\hat c_1}I_{\hat n_1}\\
                       \end{array}
                     \right).
\end{equation}

Let  \[H=\frac{1}{n}tr( A^f),\ H_1=\frac{1}{n_1}tr( A^{f_1})  \hbox{ and }\hat{H}_1 =\frac{1}{\hat{n}_1}tr( A^{\hat{f}_1})\] denote the mean curvature vectors of $f$, $f_1$ and $\hat{f}_1$ respectively.
Let $S$, $S_1$ and $\hat{S}_1$ denote the square length of the second fundamental form
of $f$, $f_1$ and $\hat{f}_1$ respectively. It is direct to obtain
\begin{lemma}\label{lemma-1}\
\begin{enumerate}
	\item The mean curvature $H$ and the square length $S$ satisfy
	\begin{equation}\label{eq-H}
	H=\frac{1}{n}\left(\frac{n_1}{c_1}H_1,\frac{\hat{n}_1}{\hat{c}_1}\hat{H}_1\right)
	+\frac{1}{n}\left(-\frac{n_1\hat c_1}{c_1}+\frac{\hat{n}_1c_1}{\hat c_1}\right)\eta,
	\end{equation}
 and
	\begin{equation}\label{eq-SS}
	S=\left|A^f\right|^2=
	\frac{n_1\hat c_1^2}{c_1^2}+\frac{\hat n_1c^2_1}{\hat c_1^2}+
	\frac{S_1}{c_1^2}+\frac{\hat S_1}{\hat c_1^2}.\end{equation}
	\item   Let $\nabla^{\perp}$, $\nabla^{\perp,f_1}$ and $\nabla^{\perp,\hat{f}_1}$ denote the normal connection of $f$, $f_1$ and $\hat{f}_1$ respectively. Then
\begin{equation}\label{eq-eta}
	\nabla^\perp \eta\equiv0.
	\end{equation}
Let $X\in\Gamma(TM)$ such that $X=X_1+\hat{X}_1$  with $X_1\in\Gamma(TM_1)$ and $\hat{X}_1\in\Gamma(T\hat{M}_1)$. Then \begin{equation}\label{eq-nabla}
 {\nabla^{\perp}_{f_\ast(X)}\left(\sigma_i,\tau_j\right)
=\left( \frac{1}{c_1}\nabla^{\perp,f_1}_{f_{1\ast}(X_1)}\sigma_i, \frac{1}{\hat c_1} \nabla^{\perp,\hat{f}_1}_{\hat{f}_{1\ast}(\hat{X}_1)}\tau_j\right).}
\end{equation}

\end{enumerate}	
\end{lemma}
\begin{proof}
	(1)  comes from the expression of $A^f$ in terms of  $A^{f_1}$ and $A^{\hat{f}_1}$ as above.
	
Consider  \eqref{eq-eta}. We have
\[\dd\eta=\left(\hat{c}_1\dd f_1,-c_1\dd \hat{f}_1\right)\perp\left\{(\sigma_i,\tau_j),1\leq i\leq p_1,1\leq j\leq \hat{p}_1\right\}\]
since $\dd f_1\perp \sigma_i$ and $\dd \hat{f}_1\perp\tau_j$. \eqref{eq-eta} follows then.
 {The equation \eqref{eq-nabla} of (2) comes from the choice of the basis of normal bundle and the fact that $f_{\ast}(X)=\frac{1}{c_1}f_{1\ast}(X_1)+\frac{1}{\hat c_1}\hat{f}_{1\ast}(\hat{X}_1)$.}
\end{proof}
\subsection{Geometric properties of product of submanifolds}
Lemma 2.1 leads to the following
\begin{theorem}\cite{TZ,CH}
	Let $f_1:M_1\rightarrow S^{n_1+p_1}$  and  $\hat f_1:\hat M_1\rightarrow S^{\hat n_1+\hat p_1}$  be two  submanifolds   of dimension $n_1$ and $\hat{n}_1$ respectively. Set $n=n_1+\hat n_1$ and $p=p_1+\hat p_1+1$.
	The $f=(c_1f_1,\hat{c}_1\hat{f}_{1}):M=M_1\times\hat  M_1\rightarrow S^{n+p}$ is a minimal submanifold if and only if both $f_1$ and $\hat{f}_1$ are minimal and  \[(c_1,\hat{c}_1)=\left(\sqrt{\frac{n_1}{n}}, \sqrt{\frac{\hat n_1}{n}}\right).\]	
\end{theorem}
Moreover, we  have the following result which could be known before.  We refer to \cite{Terng} for the study of submanifolds with flat normal bundle.
\begin{theorem}
We retain the notion as above. Then
\begin{enumerate}
	\item
	$f$ is a  submanifold with parallel mean curvature vector if and only if both $f_1$ and $\hat{f}_1$ are with parallel mean curvature vector.
	\item
	$f$ is a submanifold with flat normal bundle if and only if both $f_1$ and $\hat{f}_1$ have flat normal bundle.
\end{enumerate}
\end{theorem}
\begin{proof}
	The combination of \eqref{eq-H}, \eqref{eq-nabla} and \eqref{eq-eta} yields (1).
	
	For (2), we already see by \eqref{eq-eta} that $\eta$ is parallel. By \eqref{eq-nabla}, we see that $f$ has flat normal bundle if and only if both $f_1$ and $\hat{f}_1$ are of  flat normal bundle.
\end{proof}

\begin{example}	Since hypersurfaces of spheres are automatically with flat normal bundles, products of them give many submanifolds of spheres with flat normal bundles.
 \begin{enumerate}
 \item	 If we assume the hypersurfaces  being  minimal, we obtain many examples of minimal submanifolds with flat normal bundle in spheres. Repeating this procedure, one can produce more examples of minimal submanifolds with flat normal bundle in spheres.
 \item If we assume the hypersurfaces being of constant mean curvature, then any product of them gives a submanifold with both parallel mean curvature vector and flat normal bundle. Repeating this procedure, one can produce many examples of submanifolds with both parallel mean curvature vector and flat normal bundle. To be concrete, let $f_j:M_j^{n_j}\rightarrow S^{n_j+1}$, $1\leq j\leq k$, be constant mean curvature hypersurfaces, and let $c_j\in\mathbb R^+$, $j=1,\cdots,k$, be constants with $\sum c_j^2=1$. Then
     \[f:=\left(c_1f_1,\cdots,c_kf_k\right):M_1^{n_1}\times\cdots\times M_k^{n_k}\rightarrow S^{n+2k-1}\]
is a  submanifold with both parallel mean curvature vector and flat normal bundle in $S^{n+2k-1}$ with $n=n_1+\cdots+n_k$.
  \end{enumerate}

\end{example}

\section{On minimal product immersed by its first eigenfunctions}

\subsection{On the first eigenvalue of a minimal product}
Due to the famous Takahashi theorem,  an isometric submanifold $f:M^n\rightarrow S^{n+p}$ is minimal if and only if
 \[\Delta^f f=-n f.\]
Here $\Delta^f$ denote the Laplacian operator of $f$ with respect to the induced metric. As a consequence, the coordinate functions of an isometric minimal submanifold are eigenfunctions of the Laplacian w.r.t. the eigenvalue $n$. In this sense, the geometry of minimal submanifolds are closely related with the spectrum of the Laplacian of minimal submanifolds. Here we refer to \cite{BGM,C} for standard references of the spectrum theory of Riemannian manifolds.  We denote by
\[Spec_f=\left\{ 0<\lambda_1^f\leq\lambda_2^f\cdots \rightarrow +\infty \right\}\]
the spectrum of the Laplacian of $f$,
with
\[\Delta^f\phi^f_\alpha+\lambda^f_{\alpha}\phi^f_\alpha=0\]
for some non-zero function $\phi^f_\alpha$. The function $\phi^f_\alpha$ is called an   eigenfunction w.r.t $\lambda^f_{\alpha}$. We denote by
\[E^f_{\lambda_\alpha}:=Span_{\R}\{\phi^f_\alpha|\Delta^f\phi^f_\alpha+\lambda^f_{\alpha}\phi^f_\alpha=0\}\] to be the linear space spanned by the eigenfunctions w.r.t $\lambda^f_{\alpha}$. Note that $0$ is always an eigenvalue of $f$ with $E^f_{0}=Span_{\R}\{\phi\equiv1 \hbox{ on } M\}.$

An $n-$dimensional minimal submanifold is called immersed by its first eigenfunctions if the first eigenvalue of the Laplacian of it is equal to $n$.  Minimal submanifolds immersed by the first eigenfunctions are very important due to the fact that they play important roles in many geometric problems (See for instance \cite{Li-y}, \cite{EGJ}, \cite{ElSoufi1}, \cite{Montiel}, \cite{TY2}).

We have the following
\begin{theorem}\label{thm-first} $f$ is immersed by its first eigenfunctions if and only if both $f_1$ and $\hat f_1$ are immersed by their first eigenfunctions. That is, $\lambda_1^f=n$ if and only if both $\lambda_1^{f_1}=n_1$ and $\lambda_1^{\hat{f}_{1}}=\hat n_1$.
\end{theorem}
Theorem \ref{thm-first} is a simple application of the following
\begin{proposition} \label{prop-spec}$($ {Page 144, Proposition A.II.3 of \cite{BGM}}$)$ Let $(N,g)=(N_1\times N_2, g_1\times g_2)$ be a Riemannian manifold by Riemannian product of $(N_1,g_1)$ and $(N_2,g_2)$. Let $Spec_N$, $Spec_{N_1}$ and $Spec_{N_2}$ denote the spectrum of $\Delta^N$, $\Delta^{N_1}$ and  $\Delta^{N_2}$ respectively. Then $\Delta^N=\Delta^{N_1}+\Delta^{N_2}$. Moreover, we have
  \[Spec_N=\left\{\lambda^N|\lambda^N=\lambda^{N_1}+\lambda^{N_2} \hbox{ with }\lambda^{N_1}\in Spec_{N_1}  \hbox{ and }\lambda^{N_2}\in Spec_{N_2} \right\},
\]
and\[
E^N_{\lambda^N}=Span_{\R}\left\{( \phi^{N_1} \circ\pi_1)\cdot(\phi^{N_2}\circ\pi_2)|\phi^{N_1}\in E^{N_1}_{\lambda^{N_1}} \hbox{ and }\phi^{N_2}\in E^{N_2}_{\lambda^{N_2}} \right\}.\]

Here $\pi_i:N=N_1\times N_2\rightarrow N_i$, $i=1,2$, denotes the natural projection.
\end{proposition}

\ \\{\em Proof of Theorem \ref{thm-first}.}
For a smooth function $\varphi$ on $f$, the Laplacian  $\Delta_{f}$ of $\varphi$ is given by
\begin{equation}\label{eq-Lap}
  \Delta^{f}(\varphi)=\frac{n}{n_1} \Delta^{f_1}\varphi+ \frac{n}{\hat n_1}\Delta^{\hat f_1}\varphi.
\end{equation}
By Proposition \ref{prop-spec}, taking into account the scaling in \eqref{eq-Lap}, we see that
\begin{equation}\label{eq-la1}
	\lambda_1^f=min\left\{\frac{n}{n_1}\lambda_1^{f_1},  \frac{n}{\hat n_1}\lambda_1^{\hat{f}_{1}}\right\}.
\end{equation}
If $\lambda_1^{f_1}= n_1$ and
 $\lambda_1^{\hat{f}_{1}}=\hat n_1$, we obtain immediately that $\lambda_1^{f}=n$.
 Conversely,  if $\lambda_1^{f}=n$, then both $\frac{n}{n_1} \lambda_1^{f_1}\geq n$ and  $\frac{n}{\hat n_1}\lambda_1^{\hat{f}_{1}}\geq n$. Since
$ \lambda_1^{f_1}\leq n_1$ and
 $\lambda_1^{\hat{f}_{1}}\leq \hat n_1$, we obtain $ \lambda_1^{f_1}= n_1$ and
 $\lambda_1^{\hat{f}_{1}}=\hat n_1$.
\hfill$\Box$

This yields immediately a simple proof of the following   well-known fact.
\begin{corollary}\label{cor-CL} For all $n_j\in \mathbb Z^+$, the Clifford minimal submanifold
\begin{equation}\label{eq-CT-1}
f:=\left(\sqrt{\frac{n_1}{n}}f_1,\cdots,\sqrt{\frac{n_k}{n}}f_k\right):M^n=S^{n_1}\times\cdots\times S^{n_k}\rightarrow S^{n+k-1},\ n=\sum_{j=1}^kn_j,
\end{equation}
is a  minimal submanifold immersed by its first eigenfunctions and with flat normal bundle. Here
$f_j:S^{n_j}\rightarrow S^{n_j}$ denotes the standard $n_j-$sphere for   $j=1,\cdots,k.$
\end{corollary}
 \begin{example}
Theorem \ref{thm-first} allows us to construct many new examples of minimal submanifolds immersed by their first eigenfunctions,  with different kinds of topology. Let $f_1:S^{n_1}\rightarrow S^{n_1}\subset\R^{n_1+1}$ denotes the standard $n_1-$sphere.
 \begin{enumerate}
 \item Recall that  the minimal surface  $\tilde \tau_{3,1}$, the bipolar surface of Lawson's $\tau_{3,1}$, is a  minimal Klein bottle in $S^4$ immersed by its  first eigenfunctions \cite{Lawson,EGJ}.  Set $\hat f_1=\tilde \tau_{3,1}$. We obtain that
$f=(c_1f_1,\hat c_1\hat f_1): S^{n_1}\times\mathbb K\rightarrow S^{n_1+5} $ is an $n_1+2$ dimensional non-orientable minimal submanifold in  $S^{n_1+5}$  immersed by its  first eigenfunctions.
\item In \cite{CS}, it is proved that the Lawson minimal surface $\xi_{m,k}$ \cite{Lawson} is immersed in $S^3$ by its  first eigenfunctions. Set $\hat f_1=\xi_{m,k}$. We obtain that
$f=(c_1f_1,\hat c_1\hat f_1): S^{  n_1}\times\hat M_1\rightarrow S^{  n_1+4}$ is an $ n_1+2$ dimensional minimal submanifold  immersed by its  first eigenfunctions and with flat normal bundle. Here $\hat M_1$ is the Riemann surface $\xi_{m,k}$ of genus $mk$.
\item  It is proved in \cite{TY2} that all minimal isoparametric hypersurfaces are immersed by their first eigenfunctions. Let $\hat f_1:\hat M_1\rightarrow S^{  \hat n_1+1}$  be a $\hat n_1-$dimensional  minimal isoparametric hypersurface. Then $f=(c_1f_1,\hat c_1\hat f_1): S^{  n_1}\times\hat M_1\rightarrow S^{  n_1+\hat{n}_1+2}$ is an $(n_1+\hat n_1)$ dimensional  minimal submanifold  immersed by its  first eigenfunctions and with flat normal bundle.
 \end{enumerate}
 \end{example}

\subsection{On a problem of Tang and Yan}

In \cite{TY2}, Tang and Yan discussed briefly the influence of the co-dimension on the submanifold version of Yau's conjecture. They asked the following problem (See page 526 of \cite{TY2}):
\begin{problem} of \cite{TY2}: Let $M^d$ be an oriented, embedded, closed minimal submanifold in the unit sphere
$S^{n+1}$ with $d\geq \frac{2}{3}n+1$.
Is it true that $\lambda_1(M^d)=d$?
\end{problem}

For the discussion of  above problem, we need the following result on the first eigenvalue of focal submanifold of OT-FKM type. We refer to \cite{TY2} and reference therein for the definition of focal submanifold of OT-FKM type and more details.
\begin{proposition}\label{prop-tz} $($Proposition 1.1 of \cite{TY2}$)$ Let $M_2\subset S^{2k+3}$ be the focal submanifold of OT-FKM type
defined before with $(m_1, m_2) = (1, k)$. The following equality is valid:
\[\lambda_1(M_2) = \min\{4, 2+k\}.\]
Recall that $M_2$ has dimension $\dim M_2=k+2$ and co-dimension $k+1$.
\end{proposition}
The following examples give some partial answers to the above problem.
\begin{example}\
 \begin{enumerate}
 \item Let $f_1$ be a  focal submanifold of OT-FKM type in Proposition 1.1 of \cite{TY2} (Proposition \ref{prop-tz}) and let  $\hat f_1:S^m\rightarrow S^m$ be the totally geodesic map. Then by Proposition \ref{prop-tz} and  Theorem \ref{thm-first}, the minimal product
     \[  f=\left(\sqrt{\frac{k+2}{m+k+2}}f_1,\sqrt{\frac{m}{m+k+2}}\hat f_1\right):M_2\times S^m\rightarrow S^{m+2k+4}\]
     is an  embedded minimal submanifold {\em not } immersed by its first eigenfunctions of dimension $m+k+2$ and codimension $k+2$ when $k\geq 3$. In particular, when $k$ is fixed, the ratio $\frac{m+k+2}{m+2k+4}\rightarrow 1$ when $m\rightarrow+\infty$. Note that in this case the lowest co-dimension is $5$ since  $k$ is chosen to be $3$.
 \item Set $k\in \mathbb Z^+$, $k\geq2$. Let $T^2=\R^2/\Lambda=\mathbb C/\Lambda$ with $\Lambda=\{2\pi \mathbb Z\oplus \left(1+i\sqrt{4k^2-1}\right)\pi \mathbb Z\}$ and $z=u+iv$.  Then $f_1:T^2\rightarrow S^5\subset\mathbb C^3$,
 \begin{equation}\label{eq-t2}
f_1(u,v)=\left(\sqrt{\frac{2k^2-1}{4k^2-1}}e^{i\frac{2kv}{\sqrt{4k^2-1}}},\ \sqrt{\frac{k^2}{4k^2-1}}e^{i\left(u-\frac{v}{\sqrt{4k^2-1}}\right)},\ \sqrt{\frac{k^2}{4k^2-1}}e^{i\left( u+\frac{v}{\sqrt{4k^2-1}}\right)}\right),
 \end{equation}
 is an embedded flat minimal torus in $S^5$ (see for example \cite{Bryant,Ken} for the geometry of flat minimal tori in $S^n$), not immersed by its first eigenfunctions. In fact, for any $1\leq l<k$, $\frac{2l^2}{k^2}$ is an eigenvalue of $f_1$ with two linearly independent eigenfunctions $\cos\frac{2lv}{\sqrt{4k^2-1}}$ and $\sin\frac{2lv}{\sqrt{4k^2-1}}$.  Let $\hat f_1:S^m\rightarrow S^m$ be the totally geodesic map as above. Then by Theorem \ref{thm-first}, the minimal product
\[  f=\left(\sqrt{\frac{2}{m+2}}f_1,\sqrt{\frac{m}{m+2}}\hat f_1\right):T^2\times S^m\rightarrow S^{m+6}\]
is an  embedded minimal submanifold of dimension $m+2$ and co-dimension $4$, not immersed by its first eigenfunctions.
  \end{enumerate}
\end{example}

In view of  above examples, one may ask the revised version of Tang-Yan's problem \cite{TY2}.
\begin{problem}\
 { \begin{enumerate}
 \item  Let $M^d$ be an oriented, embedded closed minimal submanifold in the unit sphere
$S^{d+2}$. Is it true that $\lambda_1(M^d)=d$?
 \item  Let $M^d$ be an oriented, embedded closed minimal submanifold in the unit sphere
$S^{d+3}$ with $d\geq 7$. Is it true that $\lambda_1(M^d)=d$? What if $3\leq d< 7$?
  \end{enumerate}
  }
\end{problem}

\section{Morse index $\&$ Nullity of minimal product}
 In this section, we will obtain a lower bound estimate of $f$ in terms of the index of $f_1$ and $\hat f_1$ and some spectral data of them.  We will assume that the manifolds discussed in this section are  closed and oriented.

\subsection{First and second variation of volume, index and nullity}  Suppose $f_t$ is a variation of $f:M^n\rightarrow S^{n+p}$ with $\frac{\partial}{\partial t}(f_t)|_{t=0}=V$ being a section of the normal bundle of $f$, and let $A(t)$ denote the volume of $f_t$.
If   $f$ is a  minimal submanifold $f:M^n\rightarrow S^{n+p}$, it is well known \cite{Simons} that $A'(0)=0$ and
\begin{equation}\label{index-form} Q(V,V)=A''(0)=-\int_M\langle\Delta^{\perp,f}V+nV+\tilde{\mathcal{A}}^f(V),V\rangle=-\int_M\langle \mathcal{L}^f(V), V\rangle\dd A,
\end{equation}
where  we denote by $\Delta^{\perp,f}$ the Laplace operator of the normal bundle $T^\perp M$, and by
\begin{equation} \mathcal{L}^f:=\Delta^{\perp,f}+n+\tilde{\mathcal{A}}^f
\end{equation}
 the {\em Jacobi operator}, acting on sections of the normal bundle. Recall that
$\tilde{\mathcal{A}}^f:\Gamma(T^\perp M)\rightarrow \Gamma(T^\perp M)$ is defined as follow  ( here $A^f$ is the shape operator of $f$):
\[\langle\tilde{\mathcal{A}}^f(V),W\rangle=\langle A^f_V,A^f_W\rangle\]
for all sections $V, W$ of the normal bundle.
It is well-known \cite{Simons} that the Jacobi operator $\mathcal L^f$ is a strongly elliptic, self-adjoint operator acting on sections of  $T^\perp M$, with a discrete spectrum of real eigenvalues bounded from below, and of having finite-dimensional eigenspaces w.r.t. every eigenvalue.  Similar to the spectrum of the Laplacian, we will denote the spectrum of $\mathcal L^f$ by
 \[Spec^{\mathcal L^f}:=\left\{
\mu^{f}_1\leq \mu^{f}_{2}\leq \cdots\rightarrow+\infty \right\},\]
with $ \mu^{f}_{\alpha}$ satisfying
\[
\mathcal  L^f(V^f_{\alpha})+\mu^{f}_{\alpha}V^f_{\alpha}=0
\hbox{ for some } V^f_{\alpha}\in\Gamma(T^{\perp}M)\setminus\{0\},\alpha\in\mathbb Z^+.\]
We denote  by
\[\mathfrak{E}_{\mu^{f}_{\alpha}}:=Span_{\R}\left\{V^f_{\alpha}|\mathcal L^f(V^f_{\alpha})+\mu^{f}_{\alpha}V^f_{\alpha}=0\right\}\]
the eigenspace w.r.t. $\mu^{f}_{\alpha}$.

\begin{definition}\
	\begin{enumerate}
		\item The index $Ind(f)$ of the minimal immersion $f$ is defined to be the (Morse) index of the quadratic form $Q$ \eqref{index-form} associated with the second-variation of $f$: it is the maximal dimension of a subspace of normal sections on which $Q$ is negative definite, i.e.,

\begin{equation}\label{eq-def1}
Ind(f)=\sum_{\mu^{f}_{\alpha}<0}\dim \mathfrak{E}_{\mu^{f}_{\alpha}}.
\end{equation}
\item The nullity $Null(f)$ of $f$ is defined to be the dimension of the kernel ($0$-eigenspace) of $\mathcal{L}^f$, i.e.
\begin{equation}\label{eq-def1}Null(f):=\dim\mathfrak{E}_{\mu^{f}_{\alpha}}|_{\mu^{f}_{\alpha}=0}.
\end{equation}
Eigensections of the $0$-eigenspaces are called the Jacobi fields of $f$.
	\end{enumerate}
\end{definition}
Note that Killing fields of $f$ are Jacobi fields, which leads to a lower bound estimate of the nullity of $f$ (See   \cite{Simons}  for more details). Moreover, $f$ is called {\em non-degenerate} if all Jacobi fields of $f$ are Killing fields.

\subsection{Estimates of Index and Nullity for minimal products}

\subsubsection{The main results}
\begin{theorem}\label{thm-main1} Assume that $f$ is a minimal product of two oriented, closed minimal submanifolds $f_1$ and $\hat{f}_1$, with both $f_1$ and $\hat{f}_1$  being full  {and not totally geodesic}. Then we have
 		\begin{equation}\label{eq-ind2}
		Ind(f)\geq Ind(f_1)+Ind(\hat f_1)+n+p+2,
		\end{equation}
 and
		\begin{equation}\label{eq-null2}
		Null(f)\geq Null(f_1)+Null(\hat f_1)+3(n_1+p_1+1) (\hat{n}_1+\hat{p}_1+1).
		\end{equation}

\end{theorem}

It comes from the following computations  of index and nullity of $f$.
 \begin{theorem}\label{thm-main}  Assume that $f$ is a minimal product of two oriented, closed minimal submanifolds   $f_1$ and $\hat{f}_1$.

 \begin{enumerate}
 \item Let   $Ind(f)$,  $Ind(f_1)$ and $Ind(\hat f_1)$ denote the Morse index of $f$, $f_1$ and $\hat{f}_1$ respectively. Then
\begin{equation}\label{eq-ind}
  Ind(f)=Ind(f_1)+Ind(\hat f_1)+1+\mathfrak{I}_0+\mathfrak{I}_1+\hat{\mathfrak{I}}_1,
\end{equation}
where
\begin{equation}\label{eq-NN}
\begin{split}
\mathfrak{I}_0:=&\sharp\left\{\left( {\alpha}, {\beta}\right)\left|\frac{\lambda^{f_1}_{\alpha}}{n_1}
+\frac{\lambda^{\hat f_1}_{\beta}}{\hat{n}_1}<2,~~ \alpha\geq1,~\beta\geq1 \right.\right\}\\&+\sharp\left\{\ {\alpha}\left|\frac{\lambda^{f_1}_{\alpha}}{n_1}<2,~~\alpha\geq1 \right.\right\}+\sharp\left\{  {\beta} \left|\frac{\lambda^{\hat{f}_1}_{\beta}} {\hat{n}_1}<2,~~\beta\geq1  \right.\right\},
\end{split}\end{equation}
and
\begin{equation}\label{eq-NN1}
\mathfrak{I}_1=
\sharp\left\{\left( {\gamma}, {\beta}\right)\left|\frac{\mu^{f_1}_{\gamma}}{n_1}+\frac{\lambda^{\hat{f}_1}_{\beta}}{\hat{n}_1}<0
,~~\beta\geq1 \right.\right\}
,\
\hat{\mathfrak{I}}_1=\sharp\left\{\left( {\gamma}, {\alpha}\right)
\left|\frac{\mu^{\hat{f}_1}_{\gamma}}{\hat{n}_1}+\frac{\lambda^{f_1}_{\alpha}}{n_1}<0,~~\alpha\geq1 \right.\right\},
\end{equation}
with
\begin{equation}\label{eq-NN2}
\mu^{f_1}_{\gamma}\in  Spec ^{\mathcal L^{f_1}},\ \mu^{\hat f_1}_{\gamma}\in  Spec^{\mathcal L^{\hat f_1}},\ \lambda^{f_1}_{\alpha}\in Spec_{f_1} \setminus\{0\}, \ \lambda^{\hat{f}_1}_{\beta}\in Spec_{\hat{f}_1}\setminus\{0\}.
\end{equation}

In particular
\begin{equation}\label{eq-eigen}
  \mu^{f}_{1}=\min\left\{-2n,\frac{n}{n_1}\mu^{f_1}_1,\frac{n}{\hat{n}_1}\mu^{\hat{f}_1}_1\right\}\leq -2n.
\end{equation}
\item Let   $Null(f)$,  $Null(f_1)$ and $Null(\hat f_1)$ denote the nullity of $f$, $f_1$ and $\hat{f}_1$ respectively. Then
\begin{equation}\label{eq-null}
  Null(f)=Null(f_1)+Null(\hat f_1)+\mathfrak{N}_0 +\mathfrak{N}_1+\hat{\mathfrak{N}}_1,
\end{equation}
where
\begin{equation}\label{eq-NN3}
\begin{split}\mathfrak{N}_0:=&\sharp\left\{\left( {\alpha}, \beta\right)\left|\frac{\lambda^{f_1}_{\alpha}}{n_1}+\frac{\lambda^{\hat f_1}_{\beta}}{\hat{n}_1}=2 ,\alpha\geq1,\beta\geq1\right.\right\}\\
&+\sharp\left\{ {\alpha}\left|\frac{\lambda^{f_1}_{\alpha}}{n_1}=2,  \alpha\geq1 \right.\right\}+\sharp\left\{  {\beta} \left|\frac{\lambda^{\hat{f}_1}_{\beta}} {\hat{n}_1}=2, \beta\geq1\right.\right\},
\end{split}\end{equation}
and
\begin{equation}\label{eq-NN4}
\mathfrak{N}_1=\sharp\left\{\left( {\gamma}, {\beta}\right)\left|\frac{\mu^{f_1}_{\gamma}}{n_1}+\frac{\lambda^{\hat{f}_1}_{\beta}}{\hat{n}_1}=0, ~~\beta\geq1 \right.
\right\},\
\hat{\mathfrak{N}}_1=\sharp\left\{\left( {\gamma}, {\alpha}\right)\left|\frac{\mu^{\hat{f}_1}_{\gamma}}{\hat{n}_1}+\frac{\lambda^{f_1}_{\alpha}}{n_1}=0,~~\alpha\geq1 \right.
\right\},
\end{equation}
with $\{\lambda^*_{*},\ \mu^*_{*}\}$ satisfying the same conditions  as in \eqref{eq-NN2}.
\end{enumerate}
\end{theorem}

To  prove Theorem \ref{thm-main}, we need some technical lemmas.
\subsubsection{Technical lemmas}

For later computations it is natural to decompose $T^\perp M$ as
\begin{equation}\label{eq-decomp}
T^\perp M=T^{\perp,0} M\oplus T^{\perp,1} M\oplus T^{\perp,2} M\end{equation} with
 \[\begin{split}T^{\perp,0}M&=\amalg_{p\in M}T^{\perp,0}_p M=\amalg_{p\in M}\hbox{Span}_{\R}\{\eta_p\}, \\
T^{\perp,1}M&=\amalg_{p\in M}T^{\perp,1} M=\amalg_{p\in M}\hbox{Span}_{\R}\left\{\left(\sigma_i(p_1),0\right)|1\leq i\leq n_1\right\},\\
T^{\perp,2}M&=\amalg_{p\in M}T^{\perp,2}_p M=\amalg_{p\in M}\hbox{Span}_{\R}\left\{\left(0,\tau_{\hat j}(\hat p_1)\right)|1\leq \hat j\leq \hat n_1\right\}
\end{split}\]
for  any $p=(p_1,\hat p_1)\in M=M_1\times M_2$. The first lemma is
\begin{lemma}\label{lemma-m}
 {Let $\Lf$, $\Lfa$ and $\Lfb$ be the Jacobi operators of $f$, $f_1$ and $\hat f_1$ respectively.
\begin{enumerate}
\item The  decomposition \eqref{eq-decomp} is $\Lf$-invariant. That is, for any $V\in \Gamma(T^{\perp,i} M)$,  $\Lf(V)\in \Gamma(T^{\perp,i} M)$, for $i=0,1,2$.
\item Let $V\in \Gamma(T^{\perp,0} M)$ with $V=v_0\eta$. Then we have
\begin{equation}\label{eq-LLL0}
\begin{split}
\Lf(v_0\eta)=\left(\Delta^f(v_0)+2nv_0\right)\eta.\\
\end{split}
\end{equation}
 \item Let  $V\in \Gamma(T^{\perp,1} M)$ such that  $V=\sum_iv_{i}\left(\sigma_i,0\right)$ under the local basis $\{(\sigma_i,0)\}$. Then we have
\begin{equation}\label{eq-LLL1}
\begin{split}
\Lf(V)=&\frac{1}{c_1^2}\left(\Lfa\left(\sum_iv_{i}\sigma_i\right),0\right)+\frac{1}{\hat{c}_1^2}\left( \sum_i\left(\Delta^{\hat f_1}v_{i}\right)\sigma_i ,0\right).\\
\end{split}
\end{equation}
 \item Let  $V\in \Gamma(T^{\perp,2} M)$ such that  $V=\left(0,\sum_{\hat j}\hat v_{\hat j}\tau_{\hat j}\right)$ under the local basis $\{(0,\tau_{\hat j})\}$. Then we have
\begin{equation}\label{eq-LLL2}
\begin{split}
\Lf(V)= \frac{1}{\hat{c}_1^2}\left(0,\Lfb\left(\sum_{\hat j}\hat v_{\hat j}\tau_{\hat j}\right)\right) +\frac{1}{c_1^2}\left(0,\sum_{\hat j}\left(\Delta^{f_1}\hat v_{\hat j}\right)\tau_{\hat j}\right).
\end{split}
\end{equation}
\end{enumerate} }
\end{lemma}

\begin{proof} (1) comes directly from the definitions of $T^{\perp,i} M$ and the fact that  $\eta$ being parallel.

(2) By \eqref{eq-A1}, \eqref{eq-A2} and \eqref{eq-A-eta}, together with the fact that $tr A^{f}_{(\sigma_i,0)}=tr  A^{f}_{(0,\tau_j)}=0$ for all $i,j$, we obtain
\[\tilde{\mathcal A}^f(\eta)=\left(n_1\cdot\frac{\hat{c}_1^2}{c_1^2}+\hat n_1\cdot\frac{c_1^2}{\hat{c}_1^2}\right)\eta=n\eta.\]
Since $\eta$ is parallel, we have
\[\Lf(v_0\eta)=\left(\Delta^f(v_0)+nv_0\right)\eta+v_0\tilde{\mathcal A}^f(\eta)=\left(\Delta^f(v_0)+2nv_0\right)\eta.\]

(3) and (4):  Consider $\Lf\left(\sum_iv_{i}\sigma_i,\sum_{\hat j}\hat v_{\hat j}\tau_{\hat j}\right)$. First, by \eqref{eq-A1}, \eqref{eq-A2} and \eqref{eq-A-eta} we have
\[\tilde{\mathcal A}^f\left(\sum_iv_{i}\sigma_i,\sum_{\hat j}\hat v_{\hat j}\tau_{\hat j}\right)=\frac{1}{c_1^2} \left(\tilde{\mathcal A}^{f_1}\left(\sum_iv_{i}\sigma_i\right),0\right)+\frac{1}{\hat{c}_1^2}\left(0,\tilde{\mathcal A}^{\hat f_1}\left(\sum_{\hat j}\hat v_{\hat j}\tau_{\hat j}\right)\right).\]
Second, we have
\[\begin{split}\Delta^{\perp,f}\left(\sum_iv_{i}\sigma_i,\sum_{\hat j}\hat v_{\hat j}\tau_{\hat j}\right)&=\Delta^{\perp,f}\left( \sum_iv_{i}\sigma_i ,0\right)+\Delta^{\perp,f}\left(0, \sum_{\hat j}\hat v_{\hat j}\tau_{\hat j} \right)\\
&=\frac{1}{c_1^2}\left(\Delta^{\perp,f_1}\left(\sum_iv_{i}\sigma_i\right),0\right)+\frac{1}{c_1^2}\left(0, \sum_{\hat j}\left(\Delta^{f_1}\hat v_{\hat j}\right)\tau_{\hat j} \right)\\
&\ \ +\frac{1}{\hat{c}_1^2}\left( \sum_i\left(\Delta^{\hat f_1}v_{i}\right)\sigma_i ,0\right)+\frac{1}{\hat{c}_1^2}\left(0,\Delta^{\perp,\hat f_1}\left(\sum_{\hat j}\hat v_{\hat j}\tau_{\hat j}\right)\right).\\
\end{split}\]
Summing up the above equations finish the proof of \eqref{eq-LLL1} and  \eqref{eq-LLL2}.
\end{proof}

Let $\pi_1$ and $\hat{\pi}_1$ denote the  {projections} of $M$ into $M_1$ and $\hat M_1$ respectively. Then for any $\phi\in C^{\infty}(M_1)$ and $\hat{\phi}\in C^{\infty}(\hat{M}_1)$, $\phi\circ\pi_1\in C^{\infty}(M)$  and $\hat\phi\circ\hat{\pi}_1\in C^{\infty}(M)$.
Similarly, we identify  $V^{f_1}\in \Gamma(T^\perp M_1)$ as the projections defining on $M$ such that
\[V^{f_1}(p_1,\hat{p}_1):=(V^{f_1}(p_1),0)\hbox{ for all $(p_1,\hat{p}_1)\in M$.}\]
 Similarly for $V^{\hat{f}_1}\in \Gamma(T^\perp \hat{M}_1)$ we define
 \[V^{\hat{f}_1}(p_1,\hat{p}_1):=(0,V^{\hat{f}_1}(\hat{p}_1))\hbox{ for all  $(p_1,\hat{p}_1)\in M$.}\]
The second lemma is then stated as follow.
\begin{lemma}\label{lemma-m2}
The eigenspace
$\mathfrak{E}_{\mu^{f}_{\alpha}}$ of $\mu^{f}_{\alpha}$ is given by
\begin{equation}\label{eq-Ef}
\mathfrak{E}_{\mu^{f}_{\alpha}}=\mathfrak{V}_{\mu^{f}_{\alpha},0}\oplus\mathfrak{V}_{\mu^{f}_{\alpha},1}\oplus\mathfrak{V}_{\mu^{f}_{\alpha},2}
\end{equation}
with
\begin{eqnarray}
\label{eq-E11}&\mathfrak{V}_{\mu^{f}_{\alpha},0}:=&\left\{\sum_{\beta,\hat\beta}a_{\beta,\hat\beta} \left(v^{f_1}_{\beta}\circ\pi_1\right)\left(v^{\hat f_1}_{\hat\beta}\circ\hat\pi_1\right) \eta  \left|\mu^{f}_{\alpha}=n\left(\frac{\lambda^{f_1}_{\beta}}{n_1}+\frac{\lambda^{\hat f_1}_{\hat\beta}}{\hat{n}_1}-2\right), a_{\beta,\hat\beta}\in\R \right\}\right.,\\
\label{eq-E12}&\mathfrak{V}_{\mu^{f}_{\alpha},1}:=&\left\{\sum_{\gamma,\hat\beta}a_{\gamma,\hat\beta}
\left(v^{\hat{f}_1}_{\hat\beta}\circ\hat{\pi}_1\right)
\left(V^{f_1}_{\gamma},0\right)
\left|\mu^{f}_{\alpha}=n\left(\frac{\mu^{f_1}_{\gamma}}{n_1}
+\frac{\lambda^{\hat{f}_1}_{\hat\beta}}{\hat{n}_1}\right),a_{\gamma,\hat\beta}\in\R\right\}\right.,\\
\label{eq-E13}&\mathfrak{V}_{\mu^{f}_{\alpha},2}:=&\left\{\sum_{\beta,\hat\gamma}a_{\beta,\hat\gamma}\left(v^{ f_1}_{\beta}\circ\pi_1\right)\left(0,V^{\hat{f}_1}_{\hat\gamma}\right)
\left|\mu^{f}_{\alpha}=n\left(\frac{\mu^{\hat{f}_1}_{\hat\gamma}}{\hat{n}_1}
+\frac{\lambda^{f_1}_{\beta}}{n_1}\right),a_{\beta,\hat\gamma}\in\R\right\}\right.,
\end{eqnarray}
where
\[v^{f_1}_{\beta}\in E^{f_1}_{\lambda^{f_1}_{\beta}},\  \ v^{\hat f_1}_{\hat\beta}\in E^{\hat f_1}_{\lambda^{\hat f_1}_{\hat\beta}},\  V^{f_1}_{\gamma}\in\mathfrak{E} _{\mu^{f_1}_{\gamma}}\ \hbox{ and } V^{\hat{f}_1}_{\hat\gamma}\in\mathfrak{E} _{\mu^{\hat f_1}_{\hat\gamma}}.\]
\end{lemma}
\begin{proof}By Lemma  \ref{lemma-m}, we can decompose the eigenspace
$\mathfrak{E}_{\mu^{f}_{\alpha}}$ according to the decomposition of the normal bundle $T^\perp M=T^{\perp,0} M\oplus T^{\perp,1} M\oplus T^{\perp,2} M$:
\[\mathfrak{E}_{\mu^{f}_{\alpha}}
=\left(\mathfrak{E}_{\mu^{f}_{\alpha}}\cap \Gamma(T^{\perp,0} M)\right)\oplus
\left(\mathfrak{E}_{\mu^{f}_{\alpha}}\cap \Gamma(T^{\perp,1} M)\right)\oplus\left(\mathfrak{E}_{\mu^{f}_{\alpha}}\cap \Gamma(T^{\perp,2} M)\right).\]
It therefore suffices to check that
\[\mathfrak{V}_{\mu^{f}_{\alpha},l}=\mathfrak{E}_{\mu^{f}_{\alpha}}\cap \Gamma(T^{\perp,l} M),\ l=0,1,2.\]

It is direct to show that $\mathfrak{V}_{\mu^{f}_{\alpha},l}\subseteq \mathfrak{E}_{\mu^{f}_{\alpha}}\cap \Gamma(T^{\perp,l} M)$, $l=0,1,2$, from the proof of Lemma  \ref{lemma-m}.

To show that
 $\mathfrak{E}_{\mu^{f}_{\alpha}}\cap \Gamma(T^{\perp,l} M)\subseteq\mathfrak{V}_{\mu^{f}_{\alpha},l}$ for $l=0,1,2$, one needs only to check that the sections defining $\mathfrak{V}_{\mu^{f}_{\alpha},l}$ also form a basis of $\mathfrak{E}_{\mu^{f}_{\alpha}}\cap \Gamma(T^{\perp,l} M)$  for $l=0,1,2$.

Let $V\in \mathfrak{E}_{\mu^{f}_{\alpha}}\cap \Gamma(T^{\perp,0} M)$. Then $V=v_0\eta$ for some $v_0\in C^{\infty}(M)$. By
  \eqref{eq-LLL0}, we have $v_0\in E^{f}_{\lambda^{f}_{\alpha}}$ with $\lambda^{f}_{\alpha}=\mu^{f}_{\alpha}+2n$. By Proposition 3.2,
  $V=v_0\eta\in\mathfrak{V}_{\mu^{f}_{\alpha},0}$. So $\mathfrak{E}_{\mu^{f}_{\alpha}}\cap \Gamma(T^{\perp,0} M)\subseteq\mathfrak{V}_{\mu^{f}_{\alpha},0}$ holds.

Let $V\in \mathfrak{E}_{\mu^{f}_{\alpha}}\cap \Gamma(T^{\perp,1} M)$.
By  Chapter III, Theorem 5.8 of \cite{LM},\footnote{ We are thankful to the referee for pointing out Chapter III, Theorem 5.8
of \cite{LM} and for the helps on proofs of this lemma.}
 $L^2(T^{\perp}M_1)$ has a $L^2$ normalized basis $\left\{\mathcal{V}_{\gamma}^{f_1}\right\}$ via eigensections of $\mathcal L^{f_1}$, that is, $\mathcal{V}_{\gamma}^{f_1}\in\mathfrak{E} _{\mu^{f_1}_{\gamma}}$ and $||\mathcal{V}_{\gamma}^{f_1}||_{L^2}=1$ for all $\gamma$. So
 \[V=\sum_{\gamma}\left(v_{\gamma}\circ\hat\pi\right)\left(\mathcal{V}_{\gamma}^{f_1},0\right),
 \hbox{ with }v_{\gamma}= \int_{M_1}\langle V,(\mathcal{V}_{\gamma}^{f_1},0)\rangle\dd M_1.\]
 Now ${v}_{\gamma}\in C^{\infty}(\hat M_1)$
and  hence we have furthermore
 \[{v}_{\gamma}=\sum_{\hat\beta}\tilde{a}_{\gamma\hat\beta} \hat{v}^{\hat f_1}_{\hat\beta},\ \hbox{ with }\tilde{a}_{\gamma\hat\beta}\in\R\hbox{ and  }v^{\hat f_1}_{\hat\beta}\in E^{\hat f_1}_{\lambda^{\hat f_1}_{\hat\beta}}.\] Similar to computations of  $\mathcal L^f$ in  Lemma  \ref{lemma-m}, we have
\[\mathcal L^f(V)+\mu^{f}_{\alpha}V=\sum_{\gamma,\hat\beta}\tilde{a}_{\gamma\hat\beta}\left(\mu^{f}_{\alpha}-\frac{n}{n_1}{\mu^{f_1}_{\gamma}}
-\frac{n}{\hat{n}_1}{\lambda^{\hat{f}_1}_{\hat\beta}}\right)\left(\hat{v}_{\hat\beta}\circ\hat\pi\right) \left(\mathcal{V}_{\gamma}^{f_1},0\right)=0.\]
So $\tilde{a}_{\gamma\hat\beta}=0$ if $\mu^{f}_{\alpha}\neq \frac{n}{n_1}{\mu^{f_1}_{\gamma}}
+\frac{n}{\hat{n}_1}{\lambda^{\hat{f}_1}_{\hat\beta}}$. Hence $V\in\mathfrak{V}_{\mu^{f}_{\alpha},1}$ and $\mathfrak{E}_{\mu^{f}_{\alpha}}\cap \Gamma(T^{\perp,1} M)\subseteq\mathfrak{V}_{\mu^{f}_{\alpha},1}$ holds.

Similarly   $\mathfrak{E}_{\mu^{f}_{\alpha}}\cap \Gamma(T^{\perp,2} M)\subseteq\mathfrak{V}_{\mu^{f}_{\alpha},2}$ holds.
\end{proof}

\ \\
{\em Proof of Theorem \ref{thm-main}.}

(1)	By Lemma \ref{lemma-m2}, it suffices to compute $\sum_{\mu^{f}_{\alpha}<0} \left(\dim\mathfrak{V}_{\mu^{f}_{\alpha},0}+\dim \mathfrak{V}_{\mu^{f}_{\alpha},1}+\dim\mathfrak{V}_{\mu^{f}_{\alpha},2}\right)$.	We have by \eqref{eq-E11}, \eqref{eq-E12}, \eqref{eq-E13}, \eqref{eq-NN} and \eqref{eq-NN1}
\[ \sum_{\mu^{f}_{\alpha}<0} \dim\mathfrak{V}_{\mu^{f}_{\alpha},0}=\mathfrak{I}_0+1,\]
\[\sum_{\mu^{f}_{\alpha}<0}\dim \mathfrak{V}_{\mu^{f}_{\alpha},1}=\sum_{\mu^{f_1}_{\alpha}<0}\dim \mathfrak{E} _{\mu^{f_1}_{\alpha}}+\mathfrak{I}_1=Ind(f_1)+\mathfrak{I}_1,\]	
and
\[\sum_{\mu^{f}_{\alpha}<0}\dim \mathfrak{V}_{\mu^{f}_{\alpha},2}=\sum_{\mu^{\hat{f}_1}_{\alpha}<0}\dim \mathfrak{E}_{\mu^{\hat{f}_1}_{\alpha}}+\hat{\mathfrak{I}}_1=Ind(\hat{f}_1)+\hat{\mathfrak{I}}_1.\]		

(2) follows by   same computations in terms of  \eqref{eq-E11}, \eqref{eq-E12}, \eqref{eq-E13},  \eqref{eq-NN} and \eqref{eq-NN1}.
\hfill$\Box$\vspace{2mm}

 Before the proof of Theorem \ref{thm-main1}, we need one more lemma, which has been proved in \cite{Per1}  for minimal hypersurfaces.
\begin{lemma}\label{le-per}
	Let $\psi:M^n\rightarrow S^{n+p}$ be an $n-$dimensional closed minimal submanifold, which is not totally geodesic. Then 		\begin{equation}\dim \mathfrak{E}_{\mu^{\psi}_{\alpha} }|_{\mu^{\psi}_{\alpha}=-n}\geq n+p+1, 	\end{equation}
and hence
	\begin{equation}
		Ind(\psi)\geq n+p+1.
	\end{equation}
\end{lemma}
Since the index estimate of minimal submanifolds involves sections while the case of minimal hypersurfaces concerns functions, we will give a proof of Lemma \ref{le-per} in  Appendix A. It is basically taken verbatim from the proof of Lemma 3.1 of \cite{Per1}.

\ \\
{\em Proof of Theorem \ref{thm-main1}.}

(1)  Since both $f_1$ and $\hat{f}_1$ are full,  their coordinate functions are linear independent. So {there exist  some $ \alpha$ and $ \beta$} such that \[\dim E^{f_1}_{\lambda^{f_1}_{\alpha}}|_{\lambda^{f_1}_{\alpha}=n_1}\geq n_1+p_1+1,\hbox{ and }\dim E^{\hat{f}_1}_{\lambda^{\hat{f}_1}_{\beta}}|_{\lambda^{\hat{f}_1}_{\beta}=\hat{n}_1}\geq \hat{n}_1+\hat{p}_1+1.\]
So
\[\mathfrak{I}_0\geq 0+(n_1+p_1+1)+(\hat{n}_1+\hat{p}_1+1)=n+p+1.\]
This finishes \eqref{eq-ind2}.

(2)  By the same argument as above we have
\[\mathfrak{N}_0\geq \dim E^{f_1}_{\lambda^{f_1}_{\alpha}}|_{\lambda^{f_1}_{\alpha}=n_1}\cdot \dim E^{\hat{f}_1}_{\lambda^{\hat{f}_1}_{\beta}}|_{\lambda^{\hat{f}_1}_{\beta}=\hat{n}_1}\geq (n_1+p_1+1)(\hat{n}_1+\hat{p}_1+1).\]

 When $f_1$ and $\hat{f}_1$ are full and  not totally geodesic, by {Lemma \ref{le-per},} one   has
\[\dim \mathfrak{E}_{\mu^{f_1}_{\gamma} }|_{\mu^{f_1}_{\gamma}=-n_1}\geq n_1+p_1+1 \hbox{ and } \dim \mathfrak{E}_{\mu^{\hat{f}_1}_{\hat\gamma}}|_{\mu^{\hat{f}_1}_{\hat\gamma}=-\hat{n}_1}\geq \hat{n}_1+\hat{p}_1+1.\]
As a consequence
\[\mathfrak{N}_1\geq \dim \mathfrak{E}_{\mu^{f_1}_{\gamma} }|_{\mu^{f_1}_{\gamma}=-n_1}\cdot  \dim E^{\hat{f}_1}_{\lambda^{\hat{f}_1}_{\beta}}|_{\lambda^{\hat{f}_1}_{\beta}=\hat{n}_1}\geq (n_1+p_1+1)(\hat{n}_1+\hat{p}_1+1), \]
 and \[ \hat{\mathfrak{N}}_1\geq \dim \mathfrak{E}_{\mu^{\hat{f}_1}_{\hat\gamma}}|_{\mu^{\hat{f}_1}_{\hat\gamma}=-\hat{n}_1}\cdot \dim E^{f_1}_{\lambda^{f_1}_{\alpha}}|_{\lambda^{f_1}_{\alpha}=n_1}\geq (n_1+p_1+1)(\hat{n}_1+\hat{p}_1+1).\]
Summing up the above estimates yields \eqref{eq-null2}.
\hfill$\Box$

\subsubsection{Applications to some examples}

Recall that by Theorem 5.1.1 of \cite{Simons}, the totally geodesic submanifold $\hat f_1:\hat {S}^{\hat{n}_1}\rightarrow\hat {S}^{\hat{n}_1+\hat{p}_1}$ is the unique minimal submanifold with $Ind(\hat f_1)=\hat{p}_1$ or $Null(\hat f_1)=(\hat{n}_1+1)\hat{p}_1$.

Applying Theorem  \ref{thm-main}, we can re-obtain the following known results (see for example \cite{Per1} for the index of $f$).
\begin{corollary}\label{cor-CT}
	Let $f$ be the Clifford minimal hypersurface with $  f_1: S^{n_1}\rightarrow S^{n_1}$ and  $\hat f_1:\hat S^{\hat{n}_1}\rightarrow\hat S^{\hat{n}_1}$ being totally geodesic. Then
\[Ind(f)=n+3, \ Null(f)=(n_1+1)(\hat{n}_1+1).\]
In particular, the Clifford minimal hypersurface $f$ is non-degenerate.
\end{corollary}
\begin{proof}
Since $Ind(  f_1)= Ind(\hat f_1)=0$, by Theorem  \ref{thm-main}, we obtain that
\[Ind(f)=1+\mathfrak{I}_0+\mathfrak{I}_1+\hat{\mathfrak{I}}_1.\]
Recall that for the $m-$dimensional unit sphere (See e.g., page 160 of \cite{BGM})
all distinct elements of $Spec_{S^m}$ are
\begin{equation}\label{eq-Sm}\left\{0,k(k+m-1)|k\in\mathbb Z^+\right\},\end{equation}
 with {$\dim E^{S^m}_{\lambda_{1}}|_{\lambda_{1}=m}=m+1$} and  $\dim E^{S^m}_{\lambda_{\alpha}}|_{m<\lambda_{\alpha}\leq 2m}=0$. So \[\mathfrak{I}_0=n_1+1+\hat{n}_1+1=n+2, \ \mathfrak{I}_1=\hat{\mathfrak{I}}_1=0.\]
Summing up, we obtain that $Ind(f)=n+3$.

Since $Null(  f_1)= Null(\hat f_1)=0$, by Theorem  \ref{thm-main}, we  have
\[Null(f)=\mathfrak{N}_0+\mathfrak{N}_1+\hat{\mathfrak{N}}_1.\]
By \eqref{eq-Sm} and the fact that $\mu^{f_1}_1=\mu^{\hat{f}_1}_1=0$, we derive
\[\mathfrak{N}_0=(n_1+1)(\hat{n}_1+1),\ \mathfrak{N}_1=Null(f_1)=0 \hbox{ and }  \hat{\mathfrak{N}}_1=Null(\hat{f}_1)=0.\]
This finishes the computations.

Finally it is direct to compute the dimension of the space of Killing fields of $f$:
\[\dim SO(n+2)-\dim SO(n_1+1)-\dim SO(\hat n_1+1)=(n_1+1)(\hat{n}_1+1)=Null(f).\]

\end{proof}

\begin{corollary}
Let $f$ be the minimal product of a full minimal submanifold $  f_1: S^{n_1}\rightarrow S^{n_1+p_1}$ and the  totally geodesic sphere $\hat f_1:\hat S^{\hat{n}_1}\rightarrow\hat S^{\hat{n}_1}$. Then
\begin{enumerate}
\item
The index of $f$ satisfies
\begin{equation}
Ind(f)\geq Ind(f_1)+n+p_1+3,
\end{equation}
with equality holding if and only if $f_1$ is immersed by its first eigenfunctions and $\lambda^{f_1}_{n_1+p_1+2}\geq 2n_1$.
\item
The nullity of $f$ satisfies
\begin{equation}\label{eq-cor-null}
\ Null(f)\geq(n_1+1)(\hat{n}_1+1),
\end{equation}
with equality holding if and only if $f_1$ is the totally geodesic sphere $ f_1: S^{n_1}\rightarrow S^{n_1}$.
\end{enumerate} \end{corollary}
\begin{proof}
By Theorem  \ref{thm-main}, we obtain that
\[Ind(f)=Ind(f_1)+1+\mathfrak{I}_0+\mathfrak{I}_1+\hat{\mathfrak{I}}_1.\]
Similarly, we have
\[\mathfrak{I}_1 =Ind(f_1),\ \hat{\mathfrak{I}}_1=0.\]
And
\[\begin{split}
\mathfrak{I}_0=&(\hat{n}_1+1)\cdot\sum_{\lambda^{f_1}_{\alpha}<n_1}\dim E^{f_1}_{\lambda^{f_1}_{\alpha}}+\sum_{\lambda^{f_1}_{\alpha}<2n_1}\dim E^{f_1}_{\lambda^{f_1}_{\alpha}}+\hat{n}_1+1\\
&\geq n_1+p_1+1+\hat{n}_1+1\\
&=n+p_1+2.\\
\end{split}\]
So $Ind(f)\geq Ind(f_1)+n+p_1+2$, with equality holding if and only if
both $\sum_{\lambda^{f_1}_{\alpha}<n_1}\dim E^{f_1}_{\lambda^{f_1}_{\alpha}}=0 $ and $\sum_{\lambda^{f_1}_{\alpha}<2n_1}\dim E^{f_1}_{\lambda^{f_1}_{\alpha}}=n_1+p_1+1$.
That is, $f_1$ is immersed by its first eigenfunctions and $\lambda^{f_1}_{n_1+p_1+2}\geq 2n_1$.

 By Theorem  \ref{thm-main}, we obtain
\[Null(f)=Null(  f_1)+\mathfrak{N}_0+\mathfrak{N}_1+\hat{\mathfrak{N}}_1.\]
By   \eqref{eq-Sm} and the fact that $\mu^{\hat{f}_1}_1=0$, we get
\[\mathfrak{N}_0\geq (n_1+1)(\hat{n}_1+1),\ \mathfrak{N}_1\geq Null(f_1)\geq 0 \hbox{ and }  \hat{\mathfrak{N}}_1=Null(\hat{f}_1)=0.\]
Therefore \eqref{eq-cor-null} holds. And the equality holds if and only if $Null(f_1)=0$, which means exactly that \cite{Simons} $f_1$ is the totally geodesic sphere $ f_1: S^{n_1}\rightarrow S^{n_1}$.
\end{proof}

\subsection{Index $\&$ nullity of the Clifford minimal submanifolds}
In this subsection we compute the index $\&$ nullity of the Clifford minimal submanifolds.

\subsubsection{Index}
\begin{proposition}\label{prop-CT}
The Clifford minimal submanifold \eqref{eq-CT-1}
$f=\left(\sqrt{\frac{n_1}{n}}f_1,\cdots,\sqrt{\frac{n_k}{n}}f_k\right):M^n=S^{n_1}\times\cdots\times S^{n_k}\rightarrow S^{n+k-1}$  has
\begin{equation}\label{eq-ind-CT}
Ind(f)=(k-1)\sum_{j=1}^kn_j+\sum_{l=1}^{k-1}(2l+1)=(k-1)(n+k+1).
\end{equation}
  \end{proposition}

To prove Proposition \ref{prop-CT}, we need the following Lemma {which itself is of independent interest}.

\begin{lemma}\label{lemma-C}
Assume that $f_1$ and $\hat{f}_1$ are immersed by their first eigenfunctions as in Theorem \ref{thm-main}. Assume furthermore that
 the spectrums of $f_1$ and $\hat{f}_1$ satisfy the following
\begin{enumerate}
	\item $\mu^{f_1}_{\alpha_1}<\mu^{f_1}_{\alpha_1+1}=-n_1$. Here we set $\alpha_1=0$ if there exists no $\mu^{f_1}_{\alpha_1}<-n_1$;
	\item $\mu^{\hat{f}_1}_{\beta_1}<\mu^{\hat{f}_1}_{\beta_1+1}=-\hat{n}_1$. Here we set $\beta_1=0$ if there exists no  $\mu^{\hat{f}_1}_{\beta_1}<-\hat{n}_1$;
\item $\lambda^{f_1}_{\alpha}\geq 2n_1$ for all $\lambda^{f_1}_{\alpha}>n_1$;
\item $\lambda^{\hat{f}_1}_{\alpha}\geq 2\hat{n}_1$ for all $\lambda^{\hat{f}_1}_{\alpha}>\hat{n}_1$.
\end{enumerate}
Then $f$ satisfies
\begin{enumerate}
	\item $\mu^{f}_{\alpha_1+\beta_1+1}<\mu^{f}_{\alpha_1+\beta_1+2}=-n$;
	\item $\lambda^{f}_{\alpha}\geq 2n$ for all $\lambda^{f}_{\alpha}>n$.
\end{enumerate}
Moreover, we have
\begin{equation}\label{eq-ind5}
Ind(f)=Inf(f_1)+Ind(\hat{f}_1)+1+(1+\beta_1)\cdot \dim E^{f_1}_{\lambda^{f_{1}}_1}+(1+\alpha_1)\cdot \dim E^{\hat{f}_1}_{\lambda^{\hat{f}_{1}}_1}
.\end{equation}
\end{lemma}
\begin{proof}
	By \eqref{eq-E11}, \eqref{eq-E12} and \eqref{eq-E13}, for
	$\mu^f_{\alpha}<-n$, the only possibilities are the following
	\[\mu^f_{\alpha}=0+0-2n,\ \ \mu^f_{\alpha}=\frac{n}{n_1}\mu^{f_1}_{\gamma}+0<-n,\ \
\hbox{ or }~ \mu^{f}_{\alpha}=\frac{n}{\hat{n}_1}\mu^{\hat{f}_1}_{\gamma}+0<-n.\]
By assumption (1) and (2), there are exactly $(1+\alpha_1+\beta_1)$ eigenvalues of $Spec^{\mathcal L}$ being less than $-n$. The argument $\mu^{f}_{\alpha_1+\beta_1+2}=-n$  comes from the fact that $-n\in Spec^{\mathcal L^f}$ when $f$ has co-dimension at least $1$ (See \cite{Simons}).

The second argument comes from the fact that
\[\min\left\{\lambda^f_{\alpha}|\lambda^{f}_{\alpha}>n\right\}=\min\left\{\left.2n,\ \frac{n}{n_1}\lambda^{f_1}_{\alpha},\ \frac{n}{\hat{n}_1}\lambda^{\hat{f}_1}_{\beta} \right|\lambda^{f_1}_{\alpha}>n_1,\ \lambda^{\hat{f}_1}_{\beta} >\hat{n}_1\right\}.\]

With the above estimates, it is direct to compute
\[\mathfrak{I}_0=\dim E^{f_1}_{\lambda^{f_{1}}_1}+\dim E^{\hat{f}_1}_{\lambda^{\hat{f}_{1}}_1},\ \mathfrak{I}_1=\beta_1\cdot\dim E^{f_1}_{\lambda^{f_{1}}_1}
\hbox{ and } \hat{\mathfrak{I}}_1=\alpha_1\cdot \dim E^{\hat{f}_1}_{\lambda^{\hat{f}_{1}}_1}.\]
Then \eqref{eq-ind5} follows by substituting the above equations into \eqref{eq-ind}.
	\end{proof}

\
\\{\em Proof
	of Proposition \ref{prop-CT}.} We prove it by induction.	
	Consider \[f^{(1)}=\left(\sqrt{\frac{n_1}{n_1+n_2}}f_1,\sqrt{\frac{n_2}{n_1+n_2}}f_2\right).\] One checks easily that $f_1$ and $f_2$ satisfy the assumptions in Lemma \ref{lemma-C}.  So we have
	$\mu^{f^{(1)}}_{1}<\mu^{f^{(1)}}_{2}=-(n_1+n_2)=-n^{{(1)}}$, i.e., $\alpha_1(f^{(1)})=1$.
	We also have  $\lambda^{f^{(1)}}_{\alpha}\geq 2n^{{(1)}}$ for all $\lambda^{f^{(1)}}_{\alpha}>n^{{(1)}}$.  So $f^{(1)}$ satisfies
the assumptions (1) and (3) of Lemma \ref{lemma-C}.
Moreover, \eqref{eq-ind-CT}  holds for $f^{(1)}$.
	
Now we assume that $f^{(l-1)}=\left(\sqrt{\frac{n_1}{n^{(l-1)}}}f_1,\cdots,\sqrt{\frac{n_l}{n^{(l-1)}}}f_l\right)$, with $n^{(l-1)}=n_1+\cdots+n_l$, satisfies
the assumptions (1) and (3) of Lemma \ref{lemma-C} with $\alpha_1(f^{(l-1)})=l-1$
and  index
$Ind(f^{(l-1)})=(l-1)(l+1+\sum_{j=1}^{l}n_j).$ Then
\[f^{(l)}=\left(\sqrt{\frac{n_1}{n^{(l)}}}f_1,\cdots,\sqrt{\frac{n_l}{n^{(l)}}}f_l,\sqrt{\frac{n_{l+1}}{n^{(l)}}}f_{l+1}\right),\ n^{(l)}=n_1+\cdots+n_{l+1},\]  is the minimal product of $f^{(l-1)}$ and $f_{l+1}=S^{l+1}$.
 Since \[ \dim E^{f^{(l-1)}}_{\lambda^{f^{(l-1)}}_1}=l+\sum_{j=1}^{l}n_j \hbox{ and }
 \dim E^{{f}_{l+1}}_{\lambda^{{f}_{l+1}}_1}=n_{l+1}+1,\]
we obtain by Lemma \ref{lemma-C} that
 \[\begin{split}
 Ind(f^{(l)})&=(l-1)(l+1+\sum_{j=1}^{l}n_j)+0+1+(1+0)(l+\sum_{j=1}^{l}n_j)+(1+l-1)(n_{l+1}+1)\\
 &=l(l+2+  \sum_{j=1}^{l+1}n_j).
 \end{split}\]
By Lemma \ref{lemma-C}, we also have that $f^{(l)}$ satisfies the assumptions (1) and (3) in Lemma \ref{lemma-C} with $\alpha_1(f^{(l)})=l$.
Then we finish the induction.
\hfill$\Box$
\subsubsection{Nullity}\label{null}
Throughout Section \ref{null}, we denote by $f_j$ the totally geodesic sphere $f_j:S^{n_j}\rightarrow S^{n_j}$ for any $j=1,\cdots,k$.
\begin{proposition}\label{prop-CT-N}
The Clifford minimal submanifold
	$f=\left(\sqrt{\frac{n_1}{n}}f_1,\cdots,\sqrt{\frac{n_k}{n}}f_k\right):M^n=S^{n_1}\times\cdots\times S^{n_k}\rightarrow S^{n+k-1}$, $n=\sum_{j=1}^{k}n_j$, has
	\begin{equation}\label{eq-Null-CT}
Null(f)=(k-1)\sum_{1\leq i<j\leq k}(n_i+1)(n_j+1)\geq \sum_{1\leq i<j\leq k}(n_i+1)(n_j+1),\end{equation}
 with equality holding if and only if $k=2$. In particular, $f$ is degenerate when $k\geq3$, i.e., there exists Jacobi fields of $f$ which are not Killing fields of $f$.
 \end{proposition}
The proof is based on the following lemma.
\begin{lemma}\label{lemma-Null}
	Assume $f^{(l-1)}=\left(\sqrt{\frac{n_1}{n}}f_1,\cdots,\sqrt{\frac{n_l}{n}}f_l\right)$, $n=\sum_{j=1}^{l}n_j$, has
	\begin{enumerate}
		\item $\mu^{f^{(l-1)}}_{1}=-2n$, with $\dim\mathfrak{E}_{\mu^{f^{(l-1)}}_{1}}=l-1$;
		\item $\mu^{f^{(l-1)}}_{l}=- n$, with $\dim\mathfrak{E}_{\mu^{f^{(l-1)}}_{l}}=(l-1)\sum_{j=1}^{l}(n_j+1)$;
		\item $Min\{\mu^{f^{(l-1)}}_{\alpha}>-n\}=0$, with   \[\dim\mathfrak{E}_{\mu^{f^{(l-1)}}_{\alpha}}|_{\mu^{f^{(l-1)}}_{\alpha}=0}=(l-1)\sum_{1\leq i<j\leq l}(n_i+1)(n_j+1);\]
		\item $\lambda^{f^{(l-1)}}_1=n$, with \[\dim E^{f^{(l-1)}}_{\lambda_1}=\sum_{j=1}^{l}(n_j+1);\]
		\item $\lambda^{f^{(l-1)}}_{1+\sum_{j=1}^{l}(n_j+1)}=2n$, with \[\dim E^{f^{(l-1)}}_{\lambda_{1+\sum_{j=1}^{l}(n_j+1)}}
=\sum_{1\leq i<j\leq l}(n_i+1)(n_j+1).\]
	\end{enumerate}
		Then $f^{(l)}=\left(\sqrt{\frac{n_1}{\tilde{n}}}f_1,\cdots,\sqrt{\frac{n_l}{\tilde{n}}}f_l,\sqrt{\frac{n_{l+1}}{\tilde{n}}}f_{l+1}\right)$ with $\tilde{n}=n+n_{l+1}$ satisfies
	\begin{enumerate}
		\item $\mu^{f^{(l)}}_{1}=-2\tilde{n}$ with $\dim\mathfrak{E}_{\mu^{f^{(l)}}_{1}}=l$;
		\item $\mu^{f^{(l)}}_{l+1}=-\tilde{n}$
		with
		$\dim\mathfrak{E}_{\mu^{f^{(l)}}_{l}}=l\sum_{j=1}^{l+1}(n_j+1)$;
		\item $Min\{\mu^{f^{(l)}}_{\alpha}>-\tilde{n}\}=0$, with   \[\dim\mathfrak{E}_{\mu^{f^{(l)}}_{\alpha}}|_{\mu^{f^{(l)}}_{\alpha}=0}=l\sum_{1\leq i<j\leq l+1}(n_i+1)(n_j+1);\]	
		\item $\lambda^{f^{(l)}}_1=\tilde{n}$, with \[\dim E^{f^{(l)}}_{\lambda_1}=\sum_{j=1}^{l+1}(n_j+1);\]
		\item $\lambda^{f^{(l)}}_{1+\sum_{j=1}^{l+1}(n_j+1)}=2\tilde{n}$, with \[\dim E^{f^{(l)}}_{\lambda_{1+\sum_{j=1}^{l+1}(n_j+1)}}=\sum_{1\leq i<j\leq l+1}(n_i+1)(n_j+1).\]
	\end{enumerate}
\end{lemma}
\begin{proof}
(1) follows from  the proof of Lemma  \ref{lemma-C}.

(2) For $\mu^{f^{(l)}}_{\alpha}=-\tilde n$, there are 4 kinds of possibilities by Lemma \ref{lemma-m2}:  $(\lambda^{f^{(l-1)}}_{\beta},\lambda^{f_{l+1}}_{\hat\beta})=(n,0)$, or $(\lambda^{f^{(l-1)}}_{\beta},\lambda^{f_{l+1}}_{\hat\beta})=(0,n_{l+1})$, or $(\mu^{f^{(l-1)}}_{\beta},\lambda^{f_{l+1}}_{\hat\beta})=(-2n,n_{l+1})$, or $(\mu^{f^{(l-1)}}_{\beta},\lambda^{f_{l+1}}_{\hat\beta})=(-n,0)$. By assumptions (1) to (5), we obtain 	\[\dim\mathfrak{E}_{\mu^{f^{(l)}}_{l+1}}=\sum_{j=1}^{l}(n_j+1)+n_{l+1}+1+(l-1)(n_{l+1}+1)+(l-1)\sum_{j=1}^{l}(n_j+1).\]

(3) Similarly by Lemma \ref{lemma-m2}, the next eigenvalue of $\mu^{f^{(l)}}_{\alpha}$ is $0$, with 4 kinds of possibilities (Note that $\mu^{f_{l+1}}_{\hat\beta}>0$ for all $\hat\beta$):
$(\lambda^{f^{(l-1)}}_{\beta},\lambda^{f_{l+1}}_{\hat\beta})=(n,n_{l+1})$, or $(\lambda^{f^{(l-1)}}_{\beta},\lambda^{f_{l+1}}_{\hat\beta})=(2n,0 )$, or $(\mu^{f^{(l-1)}}_{\beta},\lambda^{f_{l+1}}_{\hat\beta})=(-n,n_{l+1})$, or $(\mu^{f^{(l-1)}}_{\beta},\lambda^{f_{l+1}}_{\hat\beta})=(0,0)$. By assumptions (1) to (5), we obtain 	
\[\begin{split}
\dim\mathfrak{E}_{\mu^{f^{(l)}}_{\alpha}}|_{\mu^{f^{(l)}}_{\alpha}=0}=&\sum_{j=1}^{l}(n_j+1)(n_{l+1}+1)+\sum_{1\leq i<j\leq l}(n_i+1)(n_j+1)\\
&+(l-1)\sum_{j=1}^{l}(n_j+1)(n_{l+1}+1)+(l-1)\sum_{1\leq i<j\leq l}(n_i+1)(n_j+1).\\
\end{split}\]

(4) and (5) follow easily from \eqref{eq-Lap}.

\end{proof}
{\em Proof 	of Proposition \ref{prop-CT-N}.}
	First it is direct to verify that $f^{(1)}$ satisfies (1) to (5) of Lemma \ref{lemma-Null}, by use of \eqref{eq-E11}, \eqref{eq-E12}, \eqref{eq-E13}, \eqref{eq-Lap}, \eqref{eq-la1} and \eqref{eq-Sm}.
 Then applying Lemma \ref{lemma-Null}, we finish the computation of the nullity of $f$  by induction.

The last inequality of  Proposition \ref{prop-CT-N} is due to the fact that the dimension of the space of Killing fields of $f$ is $\frac{(n+k)(n+k-1)}{2}-\frac{1}{2}\sum_{j=1}^{k}n_j\left(n_j+1\right)=\sum_{1\leq i<j\leq k}(n_i+1)(n_j+1).$
\hfill$\Box$
\section{On the average of $S$ of minimal products}

In this section we will discuss the estimate of $S$ and its average. We refer to   \cite{Per2} for some discussions for the estimate of average of $S$
of minimal hypersurfaces.
\subsection{An estimate of the average of $S$}
From \eqref{eq-SS}, we have
\begin{theorem}\
\begin{enumerate}
\item
	Let $f=(c_1f_1,\hat{c}_1\hat{f}_{1})$  be a minimal product of two oriented, closed minimal submanifolds. Then for any $a\in\R$,
	 \begin{equation}\label{eq-MP-S}
	 \int_MS^a\left( S-n\right)\dd M\geq0,\end{equation}
	 with equality holding if and only if $f$ is congruent to the Clifford hypersurface $C_{n_1,\hat{n}_1}$.  In particular, when $a=0$,  we obtain
	 	 \begin{equation}\label{eq-MP-S2}
	 	   \overline{S}=n\left(1+\frac{1}{n_1} \overline{S_1}+\frac{1}{\hat{n}_1}\overline{\hat{S}_1}\right) \geq n ,\end{equation}
	 with equality holding if and only if $f$ is congruent to  the Clifford hypersurface $C_{n_1,\hat{n}_1}$. Here
\[ \overline{S}=\frac{1}{Vol(f)}\int_M S\dd M,\ \overline{S_1}=\frac{1}{Vol(f_1)}\int_{M_1} S_1\dd M_1  \hbox{ and }\overline{\hat{S}_1}=\frac{1}{Vol(\hat f_1)}\int_{\hat M_1} \hat S_1\dd \hat M_1\]	
{denote the average of the  the square of the length of the second fundamental forms of $f$, $f_1$ and $\hat f_1$  respectively.}
\item
Let $f=(c_1f_1,\cdots, c_kf_k)$  be a minimal product of $k$ oriented, closed minimal submanifolds with $f_j$ being of dimension $n_j$ and $n=n_1+\cdots+n_k$. Then for any $a\in\R$,
	 \begin{equation}\label{eq-MP-S3}
	 \int_MS^a\left( S-(k-1)n\right)\dd M\geq0,\end{equation}
	 with equality holding if and only if $f$ is congruent to the Clifford minimal submanifold $C_{n_1,\cdots,n_k}$ \eqref{eq-CT-1}.  In particular, when $a=0$,  we obtain
	 	 \begin{equation}\label{eq-MP-S3}
	 	   \overline{S}=n\left(k-1+\sum_{j=1}^k\frac{ \overline{S_j}}{n_j}\right) \geq (k-1)n,\end{equation}
	 with equality holding if and only if $f$ is congruent to  the Clifford minimal submanifold $C_{n_1,\cdots,n_k}$.  Here
\[ \overline{S}=\frac{1}{Vol(f)}\int_M S\dd M \hbox{ and }\overline{S_j}=\frac{1}{Vol( f_j)}\int_{ M_j} S_j\dd M_j, 1\leq j\leq k.\]	
\end{enumerate}
\end{theorem}
\begin{proof}
(1) By \eqref{eq-SS}  we obtain that
\[S^{a}(S-n)={n^{a+1}}\left(1+\frac{S_1}{n_1}+\frac{\hat S_1}{\hat{n}_1}\right)^a\left(\frac{S_1}{n_1}+\frac{\hat S_1}{\hat n_1}\right)\geq0.\]
Then \eqref{eq-MP-S} holds.
	Moreover, the equality holds if and only if $S_1=\hat S_1\equiv0$. Then we conclude  that $f$ is congruent to the Clifford hypersurface $C_{n_1,\hat{n}_1}$.
\eqref{eq-MP-S2} follows directly when we set $a=0$.

The proof of (2) is the same as (1) except using the following formula instead of \eqref{eq-SS}
\begin{equation}\label{eq-Skk}
{S}=n\left(k-1+\sum_{j=1}^k\frac{ {S_j}}{n_j}\right).
\end{equation}
It is deduced  directly from \eqref{eq-SS}.

\end{proof}

Recall that the scalar curvature $R$ of the minimal submanifold $f=(c_1f_1,\hat{c}_1\hat{f}_{1})$ satisfies (\cite{Simons,PT})
\[R={n(n-1)}-S.\] It  turns out that the  scalar curvatures $R$, $R_1$ and $\hat R_1$ of $f$, $f_1$ and $\hat f_1$ satisfy the following property:
\begin{equation}\label{eq-R}
  \frac{R}{n}= \frac{R_1}{n_1}+ \frac{\hat R_1}{\hat n_1}.
\end{equation}
And the average property of $S$ in \eqref{eq-MP-S2} can also be written for average scalar curvature:
\begin{equation}\label{eq-Ra}
  \frac{\overline{R}}{n}=\frac{\int_MR\dd M}{nVol(f)}= \frac{\int_{M_1}R_1\dd M_1}{n_1Vol(f_1)}+\frac{\int_{\hat M_1}\hat R_1\dd \hat M_1}{\hat n_1Vol(\hat f_1)}
  = \frac{\overline{R_1}}{n_1}+ \frac{\overline{\hat R_1}}{\hat n_1}.
\end{equation}

\subsection{On minimal product of minimal submanifolds with constant $S$ }
Concerning the value distribution of $S$ when $S$ is constant, we have the following
\begin{theorem}
	Let $f=(c_1f_1,\hat c_1\hat f_1)$ be a minimal product of closed minimal submanifolds as above. Then
	\begin{enumerate}
		\item $S\equiv const$ if and only if both $S_1\equiv const$ and $\hat{S}_1\equiv const$.
		\item If $S\equiv const$, then $S\geq n$, with equality holding if  and only if $f$ is congruent to the Clifford hypersurface $C_{n_1,\hat{n}_1}$.
				\item If $S\equiv const$ and $n<S\leq \frac{5}{3}n$, then $S=\frac{5}{3}n$ and $f$ is congruent to the minimal product of an $(n-2)-$dimensional great sphere and the Veronese two-sphere.

	\end{enumerate}
\end{theorem}
\begin{proof}
	(1) and (2) come directly from \eqref{eq-SS}.

For (3), from \eqref{eq-SS} we obtain
\[0<\frac{S_1}{n_1}+\frac{\hat S_1}{\hat n_1}\leq \frac{2}{3}.\]
Then we obtain $S_1\leq\frac{2}{3}n_1$. By \cite{LL}, we obtain that $S_1\equiv 0$ or $S_1\equiv\frac{2}{3}n_1$ and $f_1$ is congruent to the Veronese surface (with $n_1=2$). If $S_1\equiv \frac{4}{3}$, then $\hat S_1\equiv0$ and   $\hat{f}_1$ is congruent to some $\hat{n}_1-$dimensional great sphere. If $S_1\equiv0$, we obtain that
$f_1$ is congruent to some $n_1-$dimensional great sphere, and that $0<\hat{S}_1\leq\frac{2}{3}\hat{n}_1$. Again by \cite{LL}, $\hat{S}_1\equiv\frac{4}{3}$ and $\hat{f}_1$ is congruent to the Veronese surface. This finishes (3).

\end{proof}

An immediate application of \eqref{eq-Skk} yields the following  examples of closed minimal submanifolds with constant $S$.
\begin{example}\
	\begin{enumerate}
		\item Let $f_j$ be the $n_j-$dimensional totally geodesic spheres, $1\leq j\leq k$ and let $f_j$ be the Veronese surface in $S^4$ for $k+1\leq j\leq k+r$. Then
$f=(c_1f_1,\cdots,c_{k+r}f_{k+r})$ is an $n-$dimensional closed minimal submanifold  in $S^{n+p}$ with constant
 \begin{equation}
 S=n\left(\frac{5}{3}r+k-1\right).\end{equation}
 Here  $n=n_1+\cdots+n_k+2r$, $p=3r+k-1$  and $c_j=\sqrt{\frac{n_j}{n}}$ for $1\leq j\leq k+r$.

 In particular, if $r=0$, we obtain that $S=(k-1)n$ and that $f$ is some Clifford minimal submanifold; if $r=1$, we obtain that $S= \left(\frac{2}{3}+k\right)n$ and that $f$ corresponds to a minimal product  of the Veronese surface and some Clifford minimal submanifold.
 \item Let $f_j$ be one of the $n_j-$dimensional minimal isoparametric hypersurfaces with $g_j$ distinct principal curvatures, $1\leq j\leq k$. It is well-known that $g_j$ must be one of $\{1,2,3,4,6\}$ and $S_j=(g_j-1)n_j$ (see \cite{TY2}, \cite{TXY} and \cite{TY3} and reference therein for more details).  Then
$f=(c_1f_1,\cdots,c_{k}f_{k})$ is an $n-$dimensional closed minimal submanifold  in $S^{n+p}$ with constant
 \begin{equation}
 S= \left(\tilde g-1\right)n.\end{equation}
 Here   $n=n_1+\cdots+n_k$, $\tilde g=g_1+\cdots+g_k$, $p=2k-1$ and $c_j=\sqrt{\frac{n_j}{n}}$ for $1\leq j\leq k$.
	\end{enumerate}
\end{example}

%%%%%%%%%%%%%%%%%%%%%%%%%%%%%%%%%%%%%%%%%%%%%%%%%%%%%%%%%%%%%%%%%%%%%%

%%%%%%%%%%%%%%%%%%%%%%%%%%%%%%%%%%%%%%%%%%%%%%%%%%%%%%%
%%% Acknowledgements. 致谢
%%%%%%%%%%%%%%%%%%%%%%%%%%%%%%%%%%%%%%%%%%%%%%%%%%%%%%%
{\bf Acknowledgements} {CPW was partly supported by the Project 11831005  of NSFC. PW was partly supported by the Project 11971107 of NSFC. The authors are thankful to Prof. H.Z. Li, Prof. Kusner, Prof. Z.Z. Tang, Prof. Z.X. Xie, Prof. W.J. Yan, Prof. Y.S. Zhang and Prof. E.T. Zhao for valuable discussions. The authors are thankful to Prof. H.Z. Li for introducing   to us Perdomo's paper \cite{Per2}.  The authors are thankful to Prof. Z.Z. Tang and Prof. W.J. Yan for introducing  to us the problem of \cite{TY2}. The authors are thankful to the referees for many valuable corrections and suggestions.}

%%%%%%%%%%%%%%%%%%%%%%%%%%%%%%%%%%%%%%%%%%%%%%%%%%%%%%%
%%% Conflict of interest. 作者利益声明
%%%%%%%%%%%%%%%%%%%%%%%%%%%%%%%%%%%%%%%%%%%%%%%%%%%%%%%
%\InterestConflict

%%%%%%%%%%%%%%%%%%%%%%%%%%%%%%%%%%%%%%%%%%%%%%%%%%%%%%%
%%% Supplements. 补充材料, 非必选
%%%%%%%%%%%%%%%%%%%%%%%%%%%%%%%%%%%%%%%%%%%%%%%%%%%%%%%
%\Supplements{}

%%%%%%%%%%%%%%%%%%%%%%%%%%%%%%%%%%%%%%%%%%%%%%%%%%%%%%%
%%% Reference section. 参考文献
%%% citation in the content using "some words~\cite{1,2}".
%%% ~ is needed to make the reference number is on the same line with the word before it.
%%%%%%%%%%%%%%%%%%%%%%%%%%%%%%%%%%%%%%%%%%%%%%%%%%%%%%%

%%%%%%%%%%%%%%%%%%%%%%%%%%%%%%%%%%%%%%%%%%%%%%%%%%%%%%%
%%% Appendix sections. 附录章节, 非必选
%%%%%%%%%%%%%%%%%%%%%%%%%%%%%%%%%%%%%%%%%%%%%%%%%%%%%%%
\begin{appendix}

\section{\label{sec:compint}On the index estimate of a minimal submanifold in $S^{n+p+1}$}

Lemma \ref{le-per} gives an index estimate  of an $n-$dimensional minimal submanifold in $S^{n+p}$, which is of independent interest. Here we provide a proof which is in the same spirit as Perdomo's proof for hypersurfaces. We refer to \cite{KW1} for a different proof of Lemma \ref{le-per} for minimal surfaces in spheres.

\ \\ {\em Proof of Lemma \ref{le-per}.} Consider the subspace
\begin{equation}
\mathcal{E}:=\{\mathbf{a}^{\perp}\in \Gamma(T^{\perp}M)| \mathbf{a}\in\R^{n+p+1}.\}
\end{equation}
Here $T^{\perp}M$  denotes the normal bundle of $M$.
It is well-known that \cite{Simons}
\[\mathcal{E}\subseteq \mathfrak{E}_{\mu^{\phi}_{\alpha}}|_{\mu^{\phi}_{\alpha}=-n}.\]
Therefore we need only  to prove $\dim(\mathcal{E}) = n +p+ 1.$

We argue by contradiction, that is, $\dim(\mathcal{E}) < n +p+ 1.$  Then there exists a unit vector $\mathbf{a} \in \R^{n+p+1}$ such that \[\mathbf{a}^{\perp} \equiv 0 \hbox{ on } M.\]
For each $x\in M$, let $\mathbf{a}^T (x)$ be the tangential projection of $E$ onto $T_xM$. We have
\[\mathbf{a}=\mathbf{a}^T(x)+(\cdots)x.\]
Since $\mathbf a$ is perpendicular to the normal bundle $T^{\perp}M$ at each point of $M$, $\mathbf{a}^T$ defines a vector field tangent to $M$ itself.
  Since  for any $x\neq \pm \mathbf{a}$ in $M$ the great semicircle containing $x$ and $\mathbf{a}$ is an integral curve of $\mathbf{a}^T$ passing through $x$,  it is contained in $M$. This forces $M$ to be equal to the totally geodesic great $n-$sphere $S^n$ containing $\mathbf{a}$ and tangent to $T_{\mathbf{a}}M$, which   contradicts to the assumption $M$ being not totally geodesic.
\hfill$\Box$

\end{appendix}

  	\end{document}